\documentclass[reqno]{amsart}
\usepackage{graphics}
\usepackage{latexsym}
\usepackage[dvips]{epsfig}
\usepackage{color}
\usepackage{graphicx}
\usepackage{amsmath}
\usepackage{amssymb}
\usepackage{amsfonts}
\usepackage{eufrak}
\usepackage{amstext}
\usepackage{amsopn}
\usepackage{amsbsy}
\usepackage{amscd}
\usepackage{amsxtra}
\usepackage{amsthm}
\usepackage{bm}
\usepackage{CJK}

\newcommand{\R}{{\mathbb R}}

\newcommand{\beq}{\begin{equation}}
\newcommand{\eeq}{\end{equation}}
\newcommand{\ben}{\begin{eqnarray}}
\newcommand{\een}{\end{eqnarray}}
\newcommand{\beno}{\begin{eqnarray*}}
\newcommand{\eeno}{\end{eqnarray*}}

\newtheorem{thm}{Theorem}[section]

\newtheorem{lem}[thm]{Lemma}
\newtheorem{prop}[thm]{Proposition}
\newtheorem{coro}[thm]{Corollary}
\newtheorem{rmk}[thm]{Remark}

\allowdisplaybreaks

\begin{document}

\title[Analysis of Ruptures]{Qualitative Analysis of Rupture Solutions for an MEMS Problem}
\author[J. D\'avila]{Juan D\'avila}
\address{\noindent J. D\'avila -
Departamento de Ingenier\'{\i}a Matem\'atica and CMM, Universidad de
Chile, Casilla 170 Correo 3, Santiago, Chile.}
\email{jdavila@dim.uchile.cl}

\author[K. Wang]{ Kelei Wang}
 \address{\noindent K. Wang-
 Wuhan Institute of Physics and Mathematics,
The Chinese Academy of Sciences, Wuhan 430071, China.}
\email{wangkelei@wipm.ac.cn}

\author[J. Wei]{Juncheng Wei}
\address{\noindent J. Wei -
Department Of Mathematics, Chinese University Of Hong Kong, Shatin,
Hong Kong and Department of Mathematics, University of British
Columbia, Vancouver, B.C., Canada, V6T 1Z2. }
\email{wei@math.cuhk.edu.hk}

\begin{abstract}
We prove a sharp H\"older continuity estimates of rupture sets  for
sequences of solutions of the following nonlinear problem with
negative exponent
$$ \Delta u= \frac{1}{u^p}, \ p>1, \ \mbox{in} \ \Omega .$$
As a consequence, we prove the existence of rupture solutions with isolated ruptures in a bounded convex domain in $\R^2$.

\end{abstract}



\maketitle

\date{}

\section{The setting and main results}
\setcounter{equation}{0}

Of concern is the following MEMS problem in a bounded domain $ \Omega \subset \R^n$
\begin{equation}\label{equation}
\Delta u=u^{-p} \ \mbox{in} \ \ \Omega
\end{equation}
where $p>1$.

Problem (\ref{equation}) arises in modeling   an electrostatic Micro-Electromechanical System (MEMS) device.  We refer to the books by Pelesko-Bernstein \cite{pelesko-bernstein} for physical derivations and Esposito-Ghoussoub-Guo \cite{esposito-ghoussoub-guo2} for mathematical analysis.

 Of special interest are solutions that give rise to  singularities in the equation, that is such that $u \approx 0$ in some region, which in the physical model represents a {\bf rupture} in the device. The main result of this paper is to give a sharp estimate on the H\"older continuity of solutions near the ruptures and estimates on Hausdorff dimensions of such rupture sets under  natural energy assumptions.

We now state our main results. Let $u_i$ be a sequence of positive solutions to \eqref{equation} in $B_2(0)$,
satisfying
\begin{equation}\label{energy bound}
\sup_i\int_{B_2(0)}|\nabla u_i|^2+u_i^{1-p}+u_i^2=M<+\infty.
\end{equation}
Here $B_2(0)\subset\mathbb{R}^n$ is the open ball of radius $2$.
\begin{thm}\label{thm convergence}
\

\begin{itemize}
\item $u_i$ are uniformly bounded in $C^{\frac{2}{p+1}}(\overline{B_1})$;
\item Up to subsequence, $u_i$ converges uniformly to $u_\infty$ in $B_1$, strongly in
$H^1(B_1)$, and $u_i^{-p}$ converges to $u_\infty^{-p}$ in
$L^{1}(B_1)$;
\item Outside $\{u_\infty=0\}$, $u_i$ converges to
$u_\infty$ in any $C^k$ norm;
\item $u_\infty$ is a {\em stationary} solution of \eqref{equation}.
\end{itemize}
\end{thm}
By a solution we mean that $u \in H^1$, $u^{-p}\in L^1$ and satisfies \eqref{equation} in the sense of distributions.
We say a solution $u\in H^1\cap L^{1-p}$ is stationary if for any
smooth vector field $Y$ with compact support,
\begin{equation}\label{stationary condition}
\int\left(\frac{1}{2}|\nabla u|^2
-\frac{1}{p-1}u^{1-p}\right)\mbox{div}Y -DY(\nabla u,\nabla u)=0.
\end{equation}

Next we consider the partial regularity problem for stationary
solutions.

\begin{thm}\label{thm dimension estiamte of rupture set}
Assume $u$ is a $C^{\frac{2}{p+1}}$ continuous,
stationary solution of \eqref{equation}. Then  $\{u=0\}$ is a closed set with Hausdorff dimension no more than
$n-2$. Moreover, if $n=2$, $\{u=0\}$ is a discrete set.
\end{thm}

For related estimates on the zero set of solutions see \cite{jiang-lin,guo-wei-cpaa,dupaigne-ponce-porretta,davila-ponce}. The dimension estimate in Theorem~\ref{thm dimension estiamte of rupture set} is the best compared to these previous results, although with different hypotheses.

As an application of the preceding theorems, we consider the original MEMS problem in a bounded domain
\begin{equation}
\label{rup1}
-\Delta v= \frac{\lambda}{(1-v)^p} \ \mbox{in} \ \Omega, v=0 \ \mbox{on} \ \partial \Omega
\end{equation}
where $\Omega \subset \R^n$ is a smooth bounded domain. Here rupture means $v=1$.

It is known that there exists a critical parameter $\lambda_{*}>0$ such that for $\lambda <\lambda_{*}$, problem (\ref{rup1}) has a minimal solution and for $\lambda>\lambda_{*}$ there are no positive solutions. In \cite{esposito-ghoussoub-guo}, Esposito-Ghoussoub-Guo showed that when $n \leq 7$, the extremal solution at $\lambda_{*}$ is smooth and hence there is a secondary bifurcation near $\lambda_*$. When the domain is convex, it is known that the only solutions for $\lambda$ small is the minimal solutions. Thus by Rabinowitz's bifurcation theorem \cite{rabinowitz},  there exists a sequence of $\lambda_{i}\geq c_0>0$ and  a sequence of solutions $ \{u_i=1-v_i\}$ such that $ \min u_i \to 0$.
By convexity of $\Omega$ and the moving plane method, there is a neighborhood $\Omega_\delta$ of $\partial \Omega$ such that $u_i$ remains uniformly positive in $\Omega_\delta$ (see Lemma 3.2 in \cite{guo-wei-cpaa}).
As a consequence of Theorem \ref{thm convergence}, $u_i$ are uniformly bounded in $C^{\frac{2}{p+1}} (\overline \Omega)$ and hence converges uniformly to a H\"older continuous function $ u_\infty$ with nonempty rupture set $\{ u_{\infty} =0 \}$. Applying Theorem  \ref{thm dimension estiamte of rupture set} we obtain the following result.

\begin{thm}
\label{existence}
Let $\Omega \subset \R^2$ be a convex set. Then there exists a $\lambda^{*} >0$ such that the following problem
\begin{equation}
\Delta u= \frac{\lambda^{*}}{u^p} \ \mbox{in} \ \Omega, \ u=1 \ \mbox{on} \ \partial \Omega
\end{equation}
admits a weak solution $u$ such that $u$ is H\"older continuous and the rupture set of $u$ consists a finite number of points.
\end{thm}

Theorem \ref{existence} was proved by Guo and the third author \cite{guo-wei-cpaa} under the condition that $p<3$ and that the domain has two axes of symmetries.

\medskip

The proof of the uniform H\"older estimate for positive solutions in Theorem~\ref{thm convergence}  is inspired by the work of Noris, Tavares, Terracini and Verzini \cite{NTTV}, where uniform H\"older estimates are established for a strongly competitive Schr\"odinger system. A contradiction argument leads after scaling to a globally H\"older stationary nontrivial solution of 
\begin{equation}
u \Delta u=0, \quad  u\geq 0 \ \mbox{ in} \ \R^n.
\end{equation}
 But a Liouville theorem of \cite{NTTV} says that $u$ is trivial.
The argument is carried out in Section~\ref{section unit holder}  and we give the Liouville theorem in the Appendix for completeness.
The proof of the remaining statements of Theorem~\ref{thm convergence} is given in Section~\ref{section convergence}, after some preliminaries in Section~\ref{section some tools}.
The proof actually applies to a sequence of stationary solutions having a uniform H\"older bound.
Section~\ref{section dimension reduction} contains the proof of Theorem~\ref{thm dimension estiamte of rupture set}.

\medskip

\noindent
{\bf Acknowledgment.} J. D\'avila acknowledges 
support of Fondecyt  1090167, CAPDE-Anillo ACT-125 and
Fondo Basal CMM. Kelei Wang is partially supported
 by the Joint Laboratory of CAS-Croucher in Nonlinear PDE. Juncheng Wei was supported by a GRF grant
from RGC of Hong Kong.  We thank Prof. L. Dupaigne for useful discussions.

\section{The uniform H\"{o}lder continuity}
\label{section unit holder}
\setcounter{equation}{0}

In this section we prove
\begin{thm}\label{thm uniform Holder continuity}
Let $u_i$ be a sequence of positive solutions to \eqref{equation} in
$B_4$ with
\[\sup_i\int_{B_4}u_i<+\infty.\]
Then
\[\sup_i\|u_i\|_{C^{\frac{2}{p+1}}(\overline B_1)}<+\infty.\]
\end{thm}
The remaining part of this section will be devoted to the proof of
this theorem.

Note that because $u_i$ is subharmonic and positive,
\[\sup_i\|u_i\|_{L^\infty(B_2(0))}<+\infty.\]

Take $\eta\in C^\infty(\mathbb{R}^n)$ such that $\eta\equiv 1$ in
$B_1(0)$, $\{\eta>0\}=B_2(0)$, $\eta=0$ in $\R^n \setminus B_2(0)$.
 Denote
\[\hat{u}_i=u_i\eta.\]
We will actually prove that
\[\sup_i\|\hat u_i\|_{C^{\frac{2}{p+1}}(\bar B_2(0))}<+\infty.\]
Assume this is not true.
Because $\hat{u}_i$ are smooth in $B_2$, there exist $x_i,y_i\in
B_2(0)$ such that as $i\to+\infty$,
\begin{equation}\label{2.1}
L_i=\frac {|\hat{u}_i(x_i)-\hat{u}_i(y_i)|}
{|x_i-y_i|^\frac{2}{p+1}}=\max_{x,y\in B_2(0), x\not=y}\frac
{|\hat{u}_i(x)-\hat{u}_i(y)|} {|x-y|^\frac{2}{p+1}}\to+\infty.
\end{equation}
Note that because $\hat{u}_i$ are uniformly bounded, as
$i\to+\infty$, $|x_i-y_i|\to 0$.

Denote $r_i=|x_i-y_i|$ and $z_i=(y_i-x_i)/r_i$. Since $|z_i|=1$, we
can assume that $z_i\to z_\infty\in\mathbb{S}^{n-1}$. Define
\[\widetilde{u}_i(x):=
L_i^{-1}
r_i^{-\frac{2}{p+1}}\hat{u}_i(x_i + r_i x)
=
L_i^{-1}
r_i^{-\frac{2}{p+1}} u_i(x_i + r_i x) \eta(x_i + r_i x) ,
\]
and
\[\bar{u}_i(x):=L_i^{-1}
r_i^{-\frac{2}{p+1}}u_i(x_i+r_i x)\eta(x_i).
\]
 These functions are
defined in $\Omega_i=\frac{1}{r_i}(B_2(0)-x_i)$. Note that
$\Omega_i$ converges to $\Omega_\infty$, which may be the entire
space or an half space.

We first present some facts about these rescaled functions, which
will be used below. By definition we have
\begin{equation}\label{2.02}
\widetilde{u}_i(x)=\frac{\eta(x_i+r_i x)}{\eta(x_i)}\bar{u}_i(x),
\end{equation}
and
\begin{eqnarray}\label{2.03}
\nabla\widetilde{u}_i(x)&=&\frac{r_i\nabla\eta(x_i+r_i
x)}{\eta(x_i)}\bar{u}_i(x)+\frac{\eta(x_i+r_i
x)}{\eta(x_i)}\nabla\bar{u}_i(x)\\\nonumber
&=&L_i^{-1}r_i^{\frac{p-1}{p+1}}u_i(x_i+r_i x)\nabla\eta(x_i+r_ix)
+\frac{\eta(x_i+r_i x)}{\eta(x_i)}\nabla\bar{u}_i(x)\\\nonumber
&=&\frac{\eta(x_i+r_i
x)}{\eta(x_i)}\nabla\bar{u}_i(x)+O(L_i^{-1}r_i^{\frac{p-1}{p+1}}).
\end{eqnarray}

By \eqref{2.1} and noting that $|z_i|=1$, we have
\begin{equation}\label{2.2}
1=|\widetilde{u}_i(0)-\widetilde{u}_i(z_i)|=\max_{x,y\in
\Omega_i, x\not=y}\frac {|\widetilde{u}_i(x)-\widetilde{u}_i(y)|}
{|x-y|^\frac{2}{p+1}}.
\end{equation}

 Next, because
$\eta$ is Lipschitz continuous in $\overline{B_2(0)}$, for
$x\in\Omega_i$, we have a constant $C$ which depends only on
$\sup_{B_2(0)}u_i$ and the Lipschitz constant of $\eta$, such that
\begin{eqnarray}\label{2.3}
|\widetilde{u}_i(x)-\bar{u}_i(x)|&\leq&\frac{C}{L_i
r_i^{\frac{2}{p+1}}}|\eta(x_i+r_i x)-\eta(x_i)|\\\nonumber &\leq&
CL_i^{-1}r_i^{\frac{p-1}{p+1}}|x|.
\end{eqnarray}
This converges to $0$ uniformly on any compact set of
$\Omega_\infty$ as $i\to+\infty$. By the Lipschitz continuity of
$\eta$, we also have
\begin{equation}\label{2.4}
\widetilde{u}_i(x)\leq
CL_i^{-1}r_i^{\frac{p-1}{p+1}}\text{dist}(x,\partial\Omega_i).
\end{equation}
Finally, we note that $\bar{u}_i$ satisfies
\begin{equation}\label{equation rescaled}
\Delta \bar{u}_i=\varepsilon_i\bar{u}_i^{-p}.
\end{equation}
Here $\varepsilon_i= L_i^{-p-1}\eta(x_i)^{p+1}\to 0$ as
$i\to+\infty$.

We divide the proof into two cases.

\medskip

{\bf Case 1.} $A_i:=\widetilde{u}_i(0)\to+\infty$.\\
By \eqref{2.4},
\[\text{dist}(0,\partial\Omega_i)\geq
cL_ir_i^{-\frac{p-1}{p+1}}A_i\to+\infty.\] Hence $\Omega_i$
converges to $\mathbb{R}^n$. By \eqref{2.2}, we can assume that
(after passing to a subsequence of $i$) $\widetilde{u}_i-A_i$
converges to $\bar{u}_\infty$ uniformly on any compact set of
$\mathbb{R}^n$. By \eqref{2.3}, $\bar{u}_i-A_i$ converges to the
same $\bar{u}_\infty$ uniformly on any compact set of
$\mathbb{R}^n$.

For any $R>0$, if $i$ large, \eqref{2.2} and \eqref{2.3} imply that
\[\inf_{B_R(0)}\bar{u}_i\geq\inf_{B_R(0)}\widetilde{u}_i-CL_i^{-1}r_i^{\frac{p-1}{p+1}}R
\geq
A_i-R^\frac{2}{p+1}-CL_i^{-1}r_i^{\frac{p-1}{p+1}}R\geq\frac{A_i}{2}.\]
So
\[0\leq\Delta(\bar{u}_i-A_i)\leq2^p\varepsilon_iA_i^{-p}\to0.\]
By standard $W^{2,q}$ estimates, for any $q\in(1,+\infty)$,
$\bar{u}_i-A_i$ are uniformly bounded in
$W^{2,q}_{loc}(\mathbb{R}^n)$. Then by the Sobolev embedding
theorem, for any $\alpha\in(0,1)$, $\bar{u}_i-A_i$ are uniformly
bounded in $C^{1,\alpha}_{loc}(\mathbb{R}^n)$. By letting
$i\to+\infty$ in \eqref{equation rescaled}, we see $\bar{u}_\infty$
is a harmonic function on $\mathbb{R}^n$.

By the uniform convergence of $\bar{u}_i-A_i$, we can take the limit
in \eqref{2.2} to get
\[1=|\bar{u}_\infty(0)-\bar{u}_\infty(z_\infty)|=\max_{x,y\in
\Omega_i}\frac {|\bar{u}_\infty(x)-\bar{u}_\infty(y)|}
{|x-y|^\alpha}.\] The first equality implies that $\bar{u}_\infty$
is non-constant, while the second one implies that $\bar{u}_\infty$
is globally $2/(p+1)-$H\"{o}lder continuous, hence a constant by the
Liouville theorem for harmonic functions. This is a contradiction.

\medskip

{\bf Case 2.} $A_i:=\widetilde{u}_i(0)\to A_\infty\in[0,+\infty)$.\\
By the first equality in \eqref{2.2},
\begin{equation}\label{2.5}
1\leq\widetilde{u}_i(0)+\widetilde{u}_i(z_i).
\end{equation}
 Then by \eqref{2.4},
\[cL_ir_i^{-\frac{p-1}{p+1}}\leq\text{dist}(0,\partial\Omega_i)+\text{dist}(z_i,\partial\Omega_i)
\leq 2\text{dist}(0,\partial\Omega_i)+1.
\]
So we still have $\text{dist}(0,\partial\Omega_i)\to+\infty$, and
$\Omega_\infty=\mathbb{R}^n$.

By \eqref{2.2}, we can assume that (by passing to a subsequence of
$i$) $\widetilde{u}_i$ converges to $\bar{u}_\infty$ uniformly on
any compact set of $\mathbb{R}^n$. By \eqref{2.3}, $\bar{u}_i$
converges to the same $\bar{u}_\infty$ uniformly on any compact set
of $\mathbb{R}^n$. By this uniform convergence, we can take the
limit in \eqref{2.5} to get
\[1\leq\bar{u}_\infty(0)+\bar{u}_\infty(z_\infty).\]
So the open set $D:=\{\bar{u}_\infty>0\}$ is non-empty.

In any compact set $D^\prime\subset\subset D$, there exists a
$\delta>0$ such that $\inf_{D^\prime}\bar{u}_\infty=2\delta$. Then
if $i$ large,
\[\inf_{D^\prime}\bar{u}_i\geq\delta.\]
By the same argument as in Case 1, we see
\[\Delta\bar{u}_\infty=0\ \ \ \mbox{in}\ D.\]
Hence $\bar{u}_\infty$ is smooth in $D$. In particular, if
$\{\bar{u}_\infty=0\}=\emptyset$, we can use the same argument in
Case 1 to get a contradiction.

In the following we assume $\{\bar{u}_\infty=0\}\neq\emptyset$.
Without loss of generality, assume that $\bar{u}_\infty(0)=0$.
\begin{lem}
$\bar{u}_i$ converges strongly to $\bar{u}_\infty$ in
$H^1_{loc}(\mathbb{R}^n)$. $\varepsilon_i\bar{u}_i^{1-p}$ converges
to $0$ in $L^1_{loc}(\mathbb{R}^n)$.
\end{lem}
\begin{proof}
Take a function $\eta\in C_0^\infty(\mathbb{R}^n)$.
 Testing the
equation of $\bar{u}_i$ with $\bar{u}_i\eta^2$, we get
\begin{equation}\label{2.6}
\int_{\mathbb{R}^n}|\nabla\bar{u}_i|^2\eta^2+\varepsilon_i\bar{u}_i^{1-p}\eta^2
+2\bar{u}_i\eta\nabla\bar{u}_i\nabla\eta=0.
\end{equation}

First, by applying the Cauchy inequality to the last term, we have
\[\int_{\mathbb{R}^n}|\nabla\bar{u}_i|^2\eta^2+\varepsilon_i\bar{u}_i^{1-p}\eta^2\leq
4\int_{\mathbb{R}^n}\bar{u}_i^2|\nabla\eta|^2.\] Because $\bar{u}_i$
are uniformly bounded in any compact set of $\mathbb{R}^n$,
$\bar{u}_i$ are uniformly bounded in $H^1_{loc}(\mathbb{R}^n)$. By
the uniform convergence of $\bar{u}_i$, they must converges weakly
to $\bar{u}_\infty$ in $H^1_{loc}(\mathbb{R}^n)$.

By taking limit in \eqref{2.6}, we obtain
\[
\lim_{i\to+\infty}\int_{\mathbb{R}^n}
|\nabla\bar{u}_i|^2\eta^2-|\nabla\bar{u}_\infty|^2\eta^2+\varepsilon_i\bar{u}_i^{1-p}\eta^2
=-\int_{\mathbb{R}^n}|\nabla\bar{u}_\infty|^2\eta^2
+2\bar{u}_\infty\eta\nabla\bar{u}_\infty\nabla\eta.
\]
On the other hand, take a $\sigma>0$ small so that
$\{\bar{u}_\infty=\sigma\}$ is a smooth hypersurface. Then because
$\bar{u}_\infty$ is harmonic in $\{\bar{u}_\infty>\sigma\}$,
\begin{eqnarray*}
\int_{\{\bar{u}_\infty>\sigma\}}|\nabla\bar{u}_\infty|^2\eta^2
+2\bar{u}_\infty\eta\nabla\bar{u}_\infty\nabla\eta
&=&\int_{\{\bar{u}_\infty=\sigma\}}\frac{\partial
\bar{u}_\infty}{\partial\nu}\bar{u}_\infty\eta^2\\
&=&\sigma\int_{\{\bar{u}_\infty=\sigma\}}\frac{\partial
\bar{u}_\infty}{\partial\nu}\eta^2\\
&=&\sigma\int_{\{\bar{u}_\infty>\sigma\}}\nabla
\bar{u}_\infty\nabla\eta^2\\
&=&O(\sigma).
\end{eqnarray*}
Here $\nu$ is the outward unit normal vector to
$\partial\{\bar{u}_\infty>\sigma\}$. By letting $\sigma\to0$, we see
\[\int_{\mathbb{R}^n}|\nabla\bar{u}_\infty|^2\eta^2
+2\bar{u}_\infty\eta\nabla\bar{u}_\infty\nabla\eta=0.\] Hence
\[\lim_{i\to+\infty}\int_{\mathbb{R}^n}
|\nabla\bar{u}_i|^2\eta^2-|\nabla\bar{u}_\infty|^2\eta^2+\varepsilon_i\bar{u}_i^{1-p}\eta^2=0.\qedhere\]
\end{proof}
\begin{rmk}
An essential point in this proof is the fact that
\[\bar{u}_\infty\Delta\bar{u}_\infty=0.\]
This is well defined, because $\Delta\bar{u}_\infty$ is a Radon
measure and $\bar{u}_\infty$ is continuous. From this we also get,
in the distributional sense
\begin{equation}\label{2.7}
\Delta\bar{u}_\infty^2=2|\nabla\bar{u}_\infty|^2.
\end{equation}
\end{rmk}

Because $\bar{u}_i>0$ in $\Omega_i$, it is smooth. Then by standard
domain variation calculation, for any vector field $Y\in
C_0^\infty(\Omega_i,\mathbb{R}^n)$,
\[\int_{\Omega_i}\left(\frac{1}{2}|\nabla\bar{u}_i|^2
-\frac{\varepsilon_i}{p-1}\bar{u}_i^{1-p}\right)\mbox{div}Y
-DY(\nabla\bar{u}_i,\nabla\bar{u}_i)=0.\] By the previous lemma, we
can take the limit to get
\[\int_{\R^n}\frac{1}{2}|\nabla\bar{u}_\infty|^2\mbox{div}Y
-DY(\nabla\bar{u}_\infty,\nabla\bar{u}_\infty)=0,\] for any vector
field $Y\in C_0^\infty(\mathbb{R}^n,\mathbb{R}^n)$.

Now we can apply Theorem \ref{Liouville theorem} in the appendix, which says
$\bar{u}_\infty$ is a constant. This is a contradiction because both
$\{\bar{u}_\infty>0\}$ and $\{\bar{u}_\infty=0\}$ are nonempty.

In conclusion, the assumption \eqref{2.1} does not hold. So
$\hat{u}_i$ are uniformly bounded in
$C^{\frac{2}{p+1}}(\overline{B_2})$. Since $\hat{u}_i=u_i$ in $B_1$,
this finishes the proof of Theorem \ref{thm uniform Holder
continuity}.

\section{Some tools}
\label{section some tools}
\setcounter{equation}{0}

In this section we first present some consequences of the uniform
H\"{o}lder continuity, which we will use to prove Theorems~\ref{thm convergence} and \ref{thm dimension estiamte of rupture set}. Therefore, throughout this section we assume that $u_i$ is a sequence of stationary solutions of \eqref{equation} in $B_2(0)$ satisfying\begin{align}
\label{unif holder}
\sup_i \|u_i\|_{C^{\frac{2}{p+1}}(\overline B_{3/2}(0))}< +\infty.
\end{align}
By Theorem~\ref{thm uniform Holder continuity},
this includes the case that $u_i$ are positive solutions  of \eqref{equation} in $B_2(0)$ satisfying \eqref{energy bound}.

\begin{lem}\label{lem growth of u^{-p}}
There exists a constant $C$  such that for any $i$, $x\in B_1$ and
$r\in(0,1/2)$,
\[\int_{B_r(x)}u_i^{-p}\leq Cr^{n-2\frac{p}{p+1}}.\]
\end{lem}
\begin{proof}
Take a nonnegative function $\eta\in C_0^\infty(B_{2r}(x))$ such
that $\eta\equiv1$ in $B_r(x)$ and $|\Delta\eta|\leq Cr^{-2}$. Then
\[\int u_i^{-p}\eta=\int \left(u_i-u_i(x)\right)\Delta\eta\leq Cr^{n-2+\frac{2}{p+1}}.\]
Here we have used the uniform $2/(p+1)-$H\"{o}lder continuity of
$u_i$, which implies that
\begin{equation}\label{3.1}
\sup_{B_r(x)}|u_i-u_i(x)|\leq Cr^{\frac{2}{p+1}}.\qedhere
\end{equation}
\end{proof}

\begin{lem}\label{lem growth of energy}
There exists a constant $C$ depending only on $M$, such that for any
$i$, $x\in B_1$ and $r\in(0,1/2)$,
\[\int_{B_r(x)}|\nabla u_i|^2+u_i^{1-p}\leq Cr^{n-2\frac{p-1}{p+1}}.\]
\end{lem}
\begin{proof}
First by the previous lemma and H\"{o}lder inequality,
\[\int_{B_r(x)}u_i^{1-p}\leq\left(\int_{B_r(x)}u_i^{-p}\right)^{\frac{p-1}{p}}|B_r(x)|^{\frac{1}{p}}
\leq Cr^{n-2\frac{p-1}{p+1}}.\]

Take an nonnegative function $\eta\in C_0^\infty(B_{2r}(x))$ such
that $\eta\equiv1$ in $B_r(x)$ and $|\nabla\eta|\leq 2r^{-1}$.
Testing the equation of $u_i$ with $(u_i-u_i(x))\eta^2$, we get
\[\int|\nabla u_i|^2\eta^2+u_i^{-p}(u_i-u_i(x))\eta^2=-2\int\nabla u_i\nabla\eta(u_i-u_i(x))\eta.\]
The Cauchy inequality gives
\[\int|\nabla u_i|^2\eta^2\leq\int
u_i^{-p}|u_i-u_i(x)|\eta^2+8\int|\nabla\eta|^2(u_i-u_i(x))^2.\]
 Then using the previous lemma and \eqref{3.1} we have
\begin{eqnarray*}
\int|\nabla u_i|^2\eta^2&\leq&\sup_{B_r(x)}|u_i-u_i(x)|\int
u_i^{-p}\eta^2+8\sup_{B_r(x)}|u_i-u_i(x)|^2\int|\nabla\eta|^2\\
&\leq& Cr^{n-2\frac{p-1}{p+1}}.
\end{eqnarray*}
\end{proof}

The following result holds for any $2/(p+1)-$H\"{o}lder continuous
solutions.
\begin{lem}\label{gradient estimate}
If $x\in\{u>0\}$,
\[|\nabla u(x)|\leq Cu(x)^{-\frac{p-1}{2}}.\]
\end{lem}
\begin{proof}
Denote $h^{\frac{2}{p+1}}=u(x)>0$. By the H\"{o}lder continuity,
$u\geq\frac{h^{\frac{2}{p+1}}}{2}$ in $B_{\delta h}(x)$, where
$\delta$ depends only the $C^{2/(p+1)}$ norm of $u$.  Note that we
also have $u\leq2h^{\frac{2}{p+1}}$ in $B_{\delta h}(x)$.

Define $\bar{u}(y)=h^{-\frac{2}{p+1}}u(x+hy)$. Then in
$B_\delta(0)$, $1/2\leq\bar{u}\leq 2$, and $\bar{u}$ satisfies the
equation \eqref{equation}. By standard elliptic estimates, there
exists a constant $C$ depending only on $\delta$ and $n$ so that
\[|\nabla\bar{u}(0)|\leq C.\]
Rescaling back we get the required claim.
\end{proof}
This estimates implies that $|\nabla u^{\frac{p+1}{2}}|\leq C$ in
$\{u>0\}$. Thus we get
\begin{coro}
$u^{\frac{p+1}{2}}$ is Lipschitz continuous.
\end{coro}

The next result is taken from \cite{ma-wei}, and it can be viewed as a
non-degeneracy result.
\begin{lem}\label{nondegeneracy}
There exists a constant $c$ depending only on $M$, such that for any
$i$, $x\subset B_1$ and $r\in(0,1/2)$,
\[\int_{B_r(x)}u_i\geq cr^{n+\frac{2}{p+1}}.\]
\end{lem}
\begin{proof}
By the H\"{o}lder inequality,
\[
\int_{B_r(x)}1=\int_{B_r(x)}u_i^{-\frac{p}{p+1}}u_i^{\frac{p}{p+1}}\\
\leq\left(\int_{B_r(x)}u_i^{-p}\right)^{\frac{1}{p+1}}\left(\int_{B_r(x)}u_i\right)^{\frac{p}{p+1}}.\]
Substituting Lemma \ref{lem growth of u^{-p}} into this we get the
estimate.
\end{proof}

Finally let us recall the monotonicity formula for stationary
solutions. \begin{thm}\label{thm monotonicity formula}
 Let $u$ be a stationary solution of \eqref{equation} in $B_1$. Then for any
$B_R(x)\subset B_1$ and $r\in(0,R)$,
$$E(r;x,u)=r^{-n+2\frac{p-1}{p+1}}\int_{B_r(x)}\left(\frac{1}{2}|\nabla u|^2
-\frac{1}{p-1}u^{1-p}\right)-\frac{r^{-n+2\frac{p-1}{p+1}-1}}{p+1}\int_{\partial
B_r(x)} u^2$$ is nondecreasing in $r$. Moreover, if $E(r;x,u)\equiv
const.$, then $u$ is homogeneous with respect to $x$
\[u(x+\lambda y)=\lambda^{\frac{2}{p+1}}u(x+y),\ \ y\in B_R(x),\ \lambda\in(0,1). \]
\end{thm}
\begin{proof}
By the proof in \cite{guo-wei-hausdorff}, we have
\begin{equation}\label{E^'}
\frac{d}{dr}E(r;x,u)=c(n,p)r^{2\frac{p-1}{p+1}-n}\int_{\partial
B_r(x)}\left(\frac{\partial u}{\partial
r}-\frac{2}{p+1}r^{-1}u\right)^2\geq0.
\end{equation}
This also characterizes the case of equality.
\end{proof}
By the equation we have
\[\int_{B_r(x)}|\nabla u|^2+u^{1-p}-\int_{\partial B_r(x)}uu_r=0.\]
Multiplying this with $\frac{2}{p-3}r^{2\frac{p-1}{p+1}-n}$, and
adding it into $E(r;x,u)$, we get another form for $E(r;x,u)$
\begin{eqnarray*}
E(r;x,u)&=&r^{-n+2\frac{p-1}{p+1}}\int_{B_r(x)}\left(\frac{1}{2}+\frac{2}{p-3}\right)|\nabla
u|^2 +\left(\frac{2}{p-3}-\frac{1}{p-1}\right)u^{1-p}\\
&&-\frac{1}{p-3}\frac{d}{dr}\left[r^{-n+2\frac{p-1}{p+1}}\int_{\partial
B_r(x)}u^{2}\right].
\end{eqnarray*}

\section{The convergence}
\label{section convergence}
\setcounter{equation}{0}

Let  $u_i$ be a sequence of stationary $C^{\frac{2}{p+1}}$ H\"older solutions of \eqref{equation} in $B_2(0)$ satisfying the uniform estimate \eqref{unif holder}.

Let us list the results we obtained in the previous sections.
There exists a constant $C$ independent of $i$, such that:
\begin{enumerate}
\item For any $x\in B_1$ and $r\in(0,1/2)$,
\begin{equation}\label{4.1}
\int_{B_r(x)}|\nabla u_i|^2+u_i^{1-p}\leq Cr^{n-2\frac{p-1}{p+1}}.
\end{equation}
\item For any $x\in B_1$ and $r\in(0,1/2)$,
\begin{equation}\label{4.2}
\int_{B_r(x)}u_i^{-p}\leq Cr^{n-2\frac{p}{p+1}}.
\end{equation}
\item For any $x,y\in B_1$,
\begin{equation}\label{4.3}
|u_i(x)-u_i(y)|\leq C|x-y|^{\frac{2}{p+1}}.
\end{equation}
\item For any $x\in B_1$ and $r\in(0,1/2)$,
\begin{equation}\label{4.4}
\int_{B_r(x)}u_i\geq \frac{1}{C}r^{\frac{2}{p+1}}.
\end{equation}
\end{enumerate}
By \eqref{4.3}, we can assume that, up to a subsequence of $i$,
$u_i$ converges uniformly to a function $u_\infty$ in $B_1$. Then
with \eqref{4.1}, $u_i$ are also uniformly bounded in $H^1(B_1)$,
and we can assume that it converges to $u_\infty$ weakly in
$H^1(B_1)$. By the uniform convergence, we see $u_\infty$ also
satisfies the estimate \eqref{4.3} and \eqref{4.4}.

By standard elliptic estimates, for any domain $\Omega\subset\subset
\{u_\infty>0\}\cap B_1$ and $k$, $u_i$ converges to $u_\infty$ in
$C^k(\Omega)$.

\begin{lem}\label{measure estimate of rupture set}
$H^{n-2\frac{p}{p+1}}(\{u_\infty=0\}\cap B_1)=0$.
\end{lem}
\begin{proof}
First by \eqref{4.4}, for any $x\in\{u_\infty=0\}\cap B_1$ and
$r\in(0,1/2)$,
\[\sup_{B_r(x)}u_\infty\geq cr^{\frac{2}{p+1}}.\]
Then by the H\"{o}lder continuity \eqref{4.3} for $u_\infty$, there
exists a ball $B_{\delta r}(y)\subset B_r(x)$ ($\delta$ depends on
the H\"{o}lder constant of $u_\infty$) such that
\[u_\infty\geq cr^{\frac{2}{p+1}}\ \ \mbox{in}B_{\delta r}(y).\]
In particular, $B_{\delta r}(y)\subset\{u_\infty>0\}$. This means
for any $x\in\{u_\infty=0\}\cap B_1$ and $r\in(0,1/2)$,
\[\frac{|\{u_\infty=0\}\cap B_r(x)|}{|B_r(x)|}\leq 1-c\delta.\]
By the Lebesgue differentiation theorem, $|\{u_\infty=0\}\cap
B_1|=0$.

Then because $u_i^{-p}$ converges to $u^{-p}$ uniformly in any
compact set of $\{u_\infty>0\}\cap B_1$, $u_i^{-p}$ converges to
$u^{-p}$ a.e. in $B_1$. By the Fatou lemma,
\begin{equation}\label{8}
\int_{B_1}u_\infty^{-p}\leq\liminf_{i\to+\infty}\int_{B_1}u_\infty^{-p}\leq
C.
\end{equation}

For any $\varepsilon>0$, take a maximal $\varepsilon-$separated set
$\{x_i,1\leq i\leq N\}$ of $\{u_\infty=0\}\cap B_1$. By definition,
$B_{\varepsilon/2}(x_i)$ are disjoint, and
\[\{u_\infty=0\}\cap B_1\subset\cup_i^N B_\varepsilon(x_i).\]
Note that every $B_\varepsilon(x_i)$ belongs to the
$\varepsilon-$neighborhood $\mathcal{N}_\varepsilon$ of $\{u_\infty=0\}\cap B_1$. Hence
\begin{equation}\label{9}
\sum_{i=1}^N\int_{B_{\varepsilon/2}(x_i)}u_\infty^{-p}\leq\int_{\mathcal{N}_\varepsilon}u_\infty^{-p},
\end{equation}
which goes to $0$ as $\varepsilon \to 0$. Because $x_i\in\{u_\infty=0\}$,
by \eqref{4.3},
\[\sup_{B_{\varepsilon/2}(x_i)}u_\infty\leq C\varepsilon^{\frac{2}{p+1}}.\]
Thus
\[\int_{B_{\varepsilon/2}(x_i)}u_\infty^{-p}\geq
C\varepsilon^{n-2\frac{p}{p+1}}.\] Substituting this into \eqref{9},
we see
\begin{eqnarray*}
\sum_{i=1}^N\left(\mbox{diam}(B_\varepsilon(x_i))\right)^{n-2\frac{p}{p+1}}
&\leq& C\sum_{i=1}^N\int_{B_{\varepsilon/2}(x_i)}u_\infty^{-p}\\
&\leq&C\int_{\mathcal{N}_\varepsilon}u_\infty^{-p}.
\end{eqnarray*}
By letting $\varepsilon\to 0$, we get
$H^{n-2\frac{p}{p+1}}(\{u_\infty=0\}\cap B_1)=0$.
\end{proof}
Since $u_i^{-1}$ converges to $u_\infty^{-1}$ a.e. in
$B_1$,  by passing limit in \eqref{4.1} and \eqref{4.2} and using
the Fatou lemma, we see $u_\infty$ also satisfies \eqref{4.1} and
\eqref{4.2}. (The estimate of $|\nabla u_\infty|$ is a direct
consequence of weak convergence in $H^1(B_1)$.)

\begin{lem}\label{convergence of u_i^-p}
$u_i^{-p}$ converges to $u_\infty^{-p}$ in $L^1(B_1)$.
\end{lem}
\begin{proof}
By the Fatou lemma, we always have
\[\int_{B_1}u_\infty^{-p}\leq\liminf_{i\to+\infty}\int_{B_1}u_i^{-p}.\]
Thus we only need to prove the reverse inequality
\[\int_{B_1}u_\infty^{-p}\geq\limsup_{i\to+\infty}\int_{B_1}u_i^{-p}.\]

By the previous lemma, for any $\varepsilon>0$, there exists a
covering of $\{u_\infty=0\}\cap B_1$ by $\cap_k C_k$, with
$\mbox{diam}C_k\leq\varepsilon$, and
\begin{equation}\label{10}
\sum_i(\mbox{diam}C_k)^{n-2\frac{p}{p+1}}\leq\varepsilon.
\end{equation}
For each $k$, take an $x_k\in \{u_\infty=0\}\cap B_1\cap C_k$.
Denote the open set
\[U:=\cup_k B_{\mbox{diam}C_k}(x_k).\]
$U$ is an open neighborhood of $\{u_\infty=0\}\cap B_1$. So in
$\left(\{u_\infty>0\}\cap B_1\right)\setminus U$, for all $i$ large,
$u_i^{-p}$ have a uniformly positive lower bound and they converge
to $u_\infty^{-p}$ uniformly. Hence
\begin{equation}\label{11}
\lim_{i\to+\infty}\int_{\left(\{u_\infty>0\}\cap B_1\right)\setminus
U}u_i^{-p}=\int_{\left(\{u_\infty>0\}\cap B_1\right)\setminus
U}u_\infty^{-p}.
\end{equation}

For each $i$ and $k$, by \eqref{4.2},
\[\int_{B_{\text{diam}C_k}(x_k)}u_i^{-p}\leq C(\text{diam}C_k)^{n-2\frac{p}{p+1}}.\]
Summing in $k$ and noting \eqref{10}, we see
\[\int_Uu_i^{-p}\leq\sum_k\int_{B_{\text{diam}C_k}(x_k)}u_i^{-p}\leq C\varepsilon.\]
Combined with \eqref{11}, we obtain
\[\int_{B_1}u_\infty^{-p}\geq\limsup_{i\to+\infty}\int_{B_1}u_i^{-p}-C\varepsilon.\]
Taking $\varepsilon\to 0$, we complete the proof.
\end{proof}

\begin{coro}
$u_\infty$ is a solution to \eqref{equation} in the distributional
sense.
\end{coro}

\begin{lem}
$u_i^{1-p}$ converges to $u_\infty^{1-p}$ in $L^1(B_1)$. $u_i$
converges to $u_\infty$ strongly in $H^1(B_1)$.
\end{lem}
\begin{proof}
Note that for any $t,s\geq0$, $|t^{1-p}-s^{1-p}|\leq
C(p)|s-t|(s^{-p}+t^{-p})$. Thus, by the previous lemma
\[\int_{B_1}|u_i^{1-p}-u_\infty^{1-p}|\leq C(p)\sup_{B_1}|u_i-u_\infty|
\left(\int_{B_1}u_i^{-p}+u_\infty^{-p}\right)\leq
C\sup_{B_1}|u_i-u_\infty|.\] This converges to $0$ by the uniform
convergence of $u_i$ to $u_\infty$.

By testing the equation of $u_i$ with $u_i\eta^2$, where $\eta\in
C_0^\infty(B_2)$, we have
\[\int_{B_2}|\nabla u_i|^2\eta^2+u_i^{1-p}\eta^2=\int_{B_2}u_i^2\Delta\frac{\eta^2}{2}.\]
By the strong convergence of $u_i$ in $L^2_{loc}(B_2)$, and the
convergence of $u_i^{1-p}$ proved above, we have
\[\lim_{i\to+\infty}\int_{B_2}|\nabla u_i|^2\eta^2+\int_{B_2}u_\infty^{1-p}\eta^2
=\int_{B_2}u_\infty^2\Delta\frac{\eta^2}{2}.\] Since $u_\infty\in
H^1(B_2)$ is a weak solution of \eqref{equation}, and
$u^{1-p}_\infty\in L^1_{loc}$, we also have
\[\int_{B_2}|\nabla u_\infty|^2\eta^2+u_\infty^{1-p}\eta^2
=\int_{B_2}u_\infty^2\Delta\frac{\eta^2}{2}.\] This gives
\[\lim_{i\to+\infty}\int_{B_2}|\nabla u_i|^2=\int_{B_2}|\nabla
u_\infty|^2,\] and the strong convergence of $u_i$ in $H^1(B_1)$.
\end{proof}
By this convergence, we can take limit in \eqref{stationary
condition} for $u_i$ to get the corresponding stationary condition
for $u_\infty$. This finishes the proof of Theorem \ref{thm
convergence}.

\section{Dimension reduction for stationary solutions}
\label{section dimension reduction}
\setcounter{equation}{0}

In this section we assume that $u$ is a $2/(p+1)-$H\"{o}lder
continuous, stationary solution of \eqref{equation} in $B_2$, with
\[\int_{B_2}|\nabla u|^2+u^{1-p}+u^2=M<+\infty.\]
By the results in Section 3, $u$ satisfies all of the estimates
\eqref{4.1}-\eqref{4.4}. In particular, $\{u=0\}$ is a closed set
satisfying (by Lemma \ref{measure estimate of rupture set})
\[H^{n-2+\frac{2}{p+1}}(\{u=0\})=0.\]

Assume that $u(0)=0$, for $\lambda\to 0$, define the blow up
sequence
\[u^\lambda(x)=\lambda^{-\frac{2}{p+1}}u(\lambda x).\]
By a rescaling, we see $u^\lambda$ satisfies
\eqref{4.1}-\eqref{4.4}, for all ball $B_r(x)\subset
B_{\lambda^{-1}}$. By the results established in Section 4, we can
get a subsequence of $\lambda_i\to0$, so that $u_i:=u^{\lambda_i}$
converges uniformly to a $u_\infty$ on any compact set of
$\mathbb{R}^n$. We also have
\begin{enumerate}
\item For each $R$, $u_i^{-p}$ converges to $u_\infty^{-p}$ in
$L^1(B_R)$;
\item For each $R$, $u_i^{1-p}$ converges to $u_\infty^{1-p}$ in
$L^1(B_R)$;
\item For each $R$, $u_i$ converges to $u_\infty$ in
$H^1(B_R)$;
\item $u_\infty$ is a stationary weak solution of \eqref{equation} in the
distributional sense;
\item $u_\infty$ is nonzero.
\end{enumerate}
To continue, we first note the following result.
\begin{lem}\label{convergence of rupture sets}
For any $\varepsilon>0$, if $i$ large, $\{u_i=0\}\cap B_1$ lies in
an $\varepsilon-$neighborhood of $\{u_\infty=0\}\cap B_1$.
\end{lem}
\begin{proof}
This is because $u_i$ converges to $u_\infty$ uniformly in any
compact set $\Omega^\prime\subset\subset\{u_\infty>0\}\cap B_1$.
Thus for $i$ large, $u_i>0$ in $\Omega^\prime$.
\end{proof}

Next we would like to use the monotonicity formula to explore the
information of the limit $u_\infty$.
\begin{lem}\label{limit of E(r)}
The limit $\lim_{r\to0}E(r;0,u)$ exists and is finite.
\end{lem}
\begin{proof}
In view of the monotonicity of $E(r;0,u)$, we only need to show that
as $r\to0$,  $E(r;0,u)$ has a uniform lower bound.

By Lemma \ref{lem growth of energy}, for each $r\in(0,1)$,
\[r^{2\frac{p-1}{p+1}-n}\int_{B_r}|\nabla u|^2+u^{1-p}\leq C.\]
Next, by Theorem \ref{thm uniform Holder continuity},
$\sup_{B_r}u\leq Cr^{\frac{2}{p+1}}$. Thus
\[r^{2\frac{p-1}{p+1}-n-1}\int_{\partial B_r}u^2\leq Cr^{2\frac{p-1}{p+1}-n-1+n-1+\frac{4}{p+1}}=C.\]
Substituting these into the first formulation of $E(r;0,u)$, we get
\[E(r;0,u)\geq -C.\qedhere\]
\end{proof}
By \eqref{E^'}, for any $r\in(0,1)$,
\[E(1;0,u)-E(r;0,u)=c\int_{B_1\setminus B_r}|x|^{2\frac{p-1}{p+1}-n}\left(\frac{\partial u}{\partial
r}-\frac{2}{p+1}r^{-1}u\right)^2dx.\] Together with the previous
lemma we get
\begin{coro}
\[\int_{B_1}|x|^{2\frac{p-1}{p+1}-n}\left(\frac{\partial u}{\partial
r}-\frac{2}{p+1}|x|^{-1}u\right)^2dx<+\infty.\]
\end{coro}
\begin{lem}\label{lem blow up limit}
$u_\infty$ is a homogeneous solution of \eqref{equation} on
$\mathbb{R}^n$.
\end{lem}
\begin{proof}
By the strong convergence of $u_i$ in $H^1_{loc}(\mathbb{R}^n)$, for
any $\eta\in(0,1)$,
\begin{eqnarray*}
&&\int_{B_{\eta^{-1}}\setminus
B_\eta}|x|^{2\frac{p-1}{p+1}-n}\left(\frac{\partial
u_\infty}{\partial r}-\frac{2}{p+1}r^{-1}u_\infty\right)^2dx\\
&=&\lim_{i\to+\infty}\int_{B_{\eta^{-1}}\setminus
B_\eta}|x|^{2\frac{p-1}{p+1}-n}\left(\frac{\partial
u_i}{\partial r}-\frac{2}{p+1}|x|^{-1}u_i\right)^2dx\\
&=&\lim_{i\to+\infty}\int_{B_{\eta^{-1}\lambda_i}\setminus
B_{\eta\lambda_i}}|x|^{2\frac{p-1}{p+1}-n}\left(\frac{\partial
u}{\partial r}-\frac{2}{p+1}|x|^{-1}u\right)^2dx\\
&=&0.
\end{eqnarray*}
The last one is guaranteed by the previous corollary.

This means for a.a. $x\in\mathbb{R}^n$,
\[\frac{\partial
u_\infty}{\partial r}-\frac{2}{p+1}r^{-1}u_\infty=0.\] Integrating
this in $r$, we get
\[u_\infty(x)=|x|^{\frac{2}{p+1}}u_\infty(\frac{x}{|x|}).\qedhere\]
\end{proof}

Define the density function (it may take value $-\infty$)
\[\Theta(x;u):=\lim_{r\to 0}E(r;x,u).\]
We have the following characterization of rupture points.
\begin{lem}
$x\in\{u>0\}$ if and only if $\Theta(x)=-\infty$.
\end{lem}
\begin{proof}
If $u(x)=2h>0$, by the continuity of $u$, $u>h$ in a ball
$B_{r_0}(x)$ and it is smooth here. Hence for $r<r_0$,
\[r^{2\frac{p-1}{p+1}-n}\int_{B_r(x)}|\nabla u|^2+u^{1-p}\leq Cr^{2\frac{p-1}{p+1}},\]
which goes to $0$ as $r\to 0$.

On the other hand,
\[r^{2\frac{p-1}{p+1}-n-1}\int_{\partial B_r(x)}u^2\geq h^2r^{2\frac{p-1}{p+1}-2},\]
which goes to $+\infty$ as $r\to0$. Substituting these into the
first formulation of $E(r;x,u)$ we get
\[\lim_{r\to0}E(r;x,u)=-\infty.\]

If $u(x)=0$, the same proof of Lemma \ref{limit of E(r)} gives
\[\Theta(x;u)=\lim_{r\to0}E(r;x,u)\geq-C.\qedhere\]
\end{proof}

\begin{lem}
$\Theta(x;u)$ is upper semi-continuous in $x$.
\end{lem}
\begin{proof}
Because $u\in H^1(B_2)$ and $u^{1-p}\in L^1(B_2)$, by the first
formulation of $E(r;x,u)$, $E(r;x,u)$ is a continuous function of
$x$. Then since $\Theta(x)$ is the decreasing limit of this family
of continuous functions, it is upper semi-continuous in $x$.
\end{proof}
\begin{lem}\label{lem homogeneous solutions}
Let $u$ be a homogeneous stationary solution of \eqref{equation} on
$\mathbb{R}^n$, satisfying estimates \eqref{4.1}-\eqref{4.4} for all
balls $B_r(x)$. Then for any $x\neq 0$,
$\Theta(x,u)\leq\Theta(0,u)$. Moreover, if
$\Theta(x,u)=\Theta(0,u)$, $u$ is translation invariant in the
direction $x$, i.e. for all $t\in\mathbb{R}$,
\[u(tx+\cdot)=u(\cdot)~~a.e. ~~\text{in}~~\mathbb{R}^n.\]
\end{lem}
\begin{proof}
With the help of the estimates \eqref{4.1}-\eqref{4.4}, similar to
the proof of Lemma \ref{limit of E(r)}, for any
$x_0\in\mathbb{R}^n$, there exists a constant $C$ such that
\[\lim\limits_{r\to+\infty}E(r;x_0,u)\leq C.\]
And we can define the blowing down sequence with respect to the base
point $x_0$,
\[u_\lambda(x)=\lambda^{\frac{4}{p-1}}u(x_0+\lambda x)\ \ \ \ \lambda\to+\infty.\]
Since $u$ is homogeneous with respect to $0$,
\[u_\lambda(x)=u(\lambda^{-1}x_0+x),\]
which converges to $u(x)$ as $\lambda\to+\infty$ uniformly in any
compact set of $\mathbb{R}^n$. $u_\lambda$ also converges strongly
in $H^1_{loc}(\mathbb{R}^n)$, $u_\lambda^{1-p}$ and $u_\lambda^{-p}$
converges in $L^1_{loc}(\mathbb{R}^{n})$. Then by the homogeneity of
$u$ and these convergence, we see
\begin{eqnarray*}
\Theta(0;u)=E(1;0,u)&=&\lim\limits_{\lambda\to+\infty}E(1;0,u_\lambda)\\
&=&\lim\limits_{\lambda\to+\infty}E(\lambda;x_0,u)\\
&\geq&\Theta(x_0;u).
\end{eqnarray*}
Moreover, if $\Theta(x_0;u)=\Theta(0,u)$, the above inequality
become an equality:
\[\lim\limits_{\lambda\to+\infty}E(\lambda;x_0,u)=\Theta(x_0;u).\]
This then implies that $E(\lambda;x_0,u)\equiv\Theta(x_0;u)$ for all
$\lambda>0$. By \eqref{E^'}, $u$ is homogeneous with respect to
$x_0$. Then for all $\lambda>0$,
\[u(x_0+x)=\lambda^{\frac{4}{p-1}}u(x_0+\lambda x)=u(\lambda^{-1} x_0+x).\]
By letting $\lambda\to+\infty$ and noting that
$u(\lambda^{-1}x_0+\cdot)$ are uniformly bounded in
$C^{\frac{2}{p+1}}_{loc}(\mathbb{R}^n)$, we see
\[u(x_0+\cdot)=u(\cdot)~~\text{on}~~\mathbb{R}^n.\]

Because $u$ is homogeneous with respect to $0$, a direct scaling
shows that $\Theta(tx_0;u)=\Theta(x_0;u)$ for all $t>0$, so the
above equality still holds if we replace $x_0$ by $tx_0$ for any
$t>0$. A change of variable shows this also holds for $t<0$.
\end{proof}
To finish the proof of Theorem \ref{thm dimension estiamte of
rupture set}, we also need
\begin{lem}\label{lem homogeneous solutions in dim 2}
Let $u$ be a $2/(p+1)-$H\"{o}lder continuous, homogeneous solution
of \eqref{equation} in $\mathbb{R}^2$. Then $\{u=0\}=\{0\}$.
\end{lem}
Here we only need the solution to be understood in the
distributional sense, i.e. $u^{-p}\in L^1_{loc}(\mathbb{R}^2)$.
\begin{proof}
There exists a function $\varphi(\theta)\in
C^{\frac{2}{p+1}}(\mathbb{S}^1)$ such that in the polar coordinates,
\[u(r,\theta)=r^{\frac{2}{p+1}}\varphi(\theta).\]
Then
\[
\int_{B_1}u^{-p}=\int_0^1\left(\int_{\mathbb{S}^1}\varphi(\theta)^{-p}d\theta\right)r^{-\frac{2p}{p+1}+1}dr
<+\infty.\]
 So
\[\int_{\mathbb{S}^1}\varphi(\theta)^{-p}d\theta<+\infty.\]
If there exists a $\theta_0\in\mathbb{S}^1$ such that
$\varphi(\theta_0)=0$, then
\[|\varphi(\theta)-\varphi(\theta_0)|\leq C|\theta-\theta_0|^{\frac{2}{p+1}}.\]
Hence near $\theta_0$, $\varphi^{-p}$ grows like
$|\theta-\theta_0|^{-\frac{2p}{p+1}}$. Since $\frac{2p}{p+1}>1$,
$\varphi^{-p}$ cannot be in $L^1(\mathbb{S}^1)$. This is a
contradiction and we must have $\varphi>0$ on $\mathbb{S}^1$.
\end{proof}
\begin{rmk}
Similar arguments show that there does not exist homogeneous
solutions in $\mathbb{R}^1$.
\end{rmk}

With these lemmas in hand we can apply the Federer's dimension
reduction principle (cf. Appendix A in \cite{S}) to deduce Theorem
\ref{thm dimension estiamte of rupture set}. For completeness we
present the proof in the case of $n=2$.

Assume there exists $x_i\in\{u=0\}\cap B_1$, such that $x_i\to x_0$
but $x_i\neq x_0$. Take $r_i=|x-x_i|$ and define
\[u_i(x)=r_i^{-\frac{2}{p+1}}u(x_0+r_i x).\]
After passing to a subsequence of $i$, we can assume that $u_i$
converges uniformly to a $2/(p+1)-$H\"{o}lder continuous,
homogeneous solution $u_\infty$ in any compact set of
$\mathbb{R}^2$. Since $z_i=(x_i-x_0)/r_i\in\mathbb{S}^1$, we can
also assume that $z_i\to z_\infty\in\mathbb{S}^1$. By the uniform
convergence of $u_i$,
\[u_\infty(z_\infty)=\lim_{i\to+\infty}u_i(z_i)=0.\]
However, Lemma \ref{lem homogeneous solutions in dim 2} says
$u_\infty>0$ outside the origin. This is a contradiction and
$\{u=0\}\cap B_1$ must be a discrete set.

\appendix
\section{A Liouville theorem}
\setcounter{equation}{0}

In this appendix we recall a Liouville theorem proved in \cite{NTTV}.
\begin{thm}\label{Liouville theorem}
Let $\alpha\in(0,1)$. Assume $\bar{u}_\infty\geq0$ is a globally
$C^\alpha(\R^n)$ continuous function, satisfying
\begin{equation}\label{limit equation}
\bar{u}_\infty\Delta\bar{u}_\infty=0 \quad\text{in }\R^n,
\end{equation}
and that $\bar{u}_\infty$
is stationary,
i.e.\[\int_{\R^n}\frac{1}{2}|\nabla\bar{u}_\infty|^2\mbox{div}Y
-DY(\nabla\bar{u}_\infty,\nabla\bar{u}_\infty)=0,\] for any vector
field $Y\in C_0^\infty(\mathbb{R}^n,\mathbb{R}^n)$. Then
$\bar{u}_\infty$ is constant.
\end{thm}

Equation \eqref{limit equation} implies that
\begin{equation}\label{limit equation 2}
\Delta\bar{u}_\infty^2=2|\nabla\bar{u}_\infty|^2 ,
\end{equation}
in the distributional sense.
Moreover, $\bar{u}_\infty$ is harmonic in the open set
$\{\bar{u}_\infty>0\}$. So if $\bar{u}_\infty>0$ everywhere, it is a
harmonic function on $\mathbb{R}^n$. Then because $\bar{u}_\infty$
is globally $C^\alpha$, the standard Liouville theorem implies that
it is constant.

In the following we assume $\{\bar{u}_\infty=0\}\neq\emptyset$.
First we present some monotonicity formulas.
\begin{prop}\label{monotonocity 1}
For $r>0$ and $x\in\mathbb{R}^n$,
$$D(r;x):=r^{2-n}\int_{B_r(x)}|\nabla \bar{u}_\infty|^2$$
is nondecreasing in $r$.
\end{prop}
\begin{proof}
For a proof, see \cite{C-L 2} Lemma 2.1. In fact by the stationary
condition, we have
\[(n-2)\int_{B_r(x)}|\nabla\bar{u}_\infty|^2=r\int_{\partial B_r(x)}|\nabla\bar{u}_\infty|^2
-2(\frac{\partial\bar{u}_\infty}{\partial r})^2.\] Then direct
calculations give
\begin{equation}\label{A.1}
\frac{d}{dr}D(r;x)=2r^{2-n}\int_{\partial
B_r(x)}\left(\frac{\partial\bar{u}_\infty}{\partial
r}\right)^2\geq0.
\end{equation}
\end{proof}
Next let $H(r;x):=r^{1-n}\int_{\partial B_r}\bar{u}_\infty^2$. By
\eqref{limit equation 2}, direct calculations give
\begin{eqnarray}\label{A8}
\frac{dH}{dr}=2r^{1-n}\int_{\partial
B_r}\bar{u}_\infty\frac{\partial\bar{u}_\infty}{\partial
r}&=&2r^{1-n}\int_{
B_r}\bar{u}_\infty\Delta\bar{u}_\infty\\\nonumber
&=&\frac{2}{r}D(r).
\end{eqnarray}
Then we get
\begin{prop}\label{monotonocity}
{\bf (Almgren monotonicity formula.)} For $r>0$ and
$x\in\mathbb{R}^n$,
$$N(r;x):=\frac{D(r;x)}{H(r;x)}$$
is nondecreasing in $r$. Moreover, if $N(r;x)\equiv d$, then
\[\bar{u}_\infty(x+ry)=r^d\bar{u}_\infty(x+y).\]
\end{prop}
\begin{proof}
Without loss of generality, take $x=0$.
\begin{eqnarray*}
\frac{d}{dr}N(r) &=&\frac{H(r)\left[2r^{2-n}\int_{\partial
B_r}\left(\frac{\partial\bar{u}_\infty}{\partial
r}\right)^2\right]-D(r)
\left(2r^{1-n}\int_{\partial B_r} \bar{u}_\infty\frac{\partial\bar{u}_\infty}{\partial r}\right)}{H(r)^2}\\
&=&2r^{3-2n}\frac{\int_{\partial B_r}\bar{u}_\infty^2\int_{\partial
B_r}\left(\frac{\partial\bar{u}_\infty}{\partial r}\right)^2-
\left(\int_{\partial B_r} \bar{u}_\infty\frac{\partial\bar{u}_\infty}{\partial r}\right)^2}{H(r)^2}\\
&\geq&0.
\end{eqnarray*}
If $N(r)\equiv d$, for any $r$,
\[\int_{\partial B_r}\bar{u}_\infty^2\int_{\partial
B_r}\left(\frac{\partial\bar{u}_\infty}{\partial r}\right)^2-
\left(\int_{\partial B_r}
\bar{u}_\infty\frac{\partial\bar{u}_\infty}{\partial
r}\right)^2=0.\] By the characterization of the equality case of the
Cauchy inequality,  there exists a $\lambda(r)$ such that
\[\frac{\partial\bar{u}_\infty}{\partial
r}=\lambda(r)\bar{u}_\infty.\] Integrating in $r$ we get a function
$\varphi(r)$ such that
\[\bar{u}_\infty(y)=\varphi(|y|)\bar{u}_\infty(\frac{y}{|y|}).\]
Then a direct calculation shows that $\varphi(|y|)=|y|^d$.
\end{proof}
\begin{prop}
If $N(r_0;x)\geq d$, then for $r>r_0$,
$$r^{1-n-2d}\int_{\partial B_r(x)}\bar{u}_\infty^2$$
is nondecreasing in $r$.
\end{prop}
\begin{proof}
Direct calculation using \eqref{A8} shows
\begin{eqnarray*}
&&\frac{d}{dr}\left(r^{1-n-2d}\int_{\partial B_r(x)}u^2+v^2\right)\\
&=&
-2dr^{-n-2d}\int_{\partial B_r(x)}(u^2+v^2)+2r^{1-n-2d}\int_{B_r(x)}|\nabla u|^2+|\nabla v|^2+2u^2v^2\\
&\geq& 0.
\end{eqnarray*}
Here we have used Proposition \ref{monotonocity}, in particular, the
fact that $N(r)\geq d$ for every $r\geq r_0$.
\end{proof}
Because $\bar{u}_\infty$ is globally $C^{\alpha}$,
\[\bar{u}_\infty(x)\leq C(1+|x|^{\alpha})\ \ \mbox{in}\ \mathbb{R}^n.\]
Hence for any $x$ and $r$ large,
$$\int_{\partial B_r(x)}\bar{u}_\infty^2\leq Cr^{n-1+2\alpha}.$$
Combining this with the previous proposition we get
\begin{equation}\label{A.001}
N(r;x)\leq\alpha,\ \ \mbox{for any}\ r>0,x\in\mathbb{R}^n.
\end{equation}

The next result is the so called ``doubling property".
\begin{prop}\label{doubling property}
Let $x\in\{\bar{u}_\infty=0\}$ and $R>0$ such that $N(R;x)\leq d$,
then for every $0<r\leq R$
\begin{equation}\label{eq:h_monotone}
H(r;x)\geq H(R;x)r^{2d}.
\end{equation}
\end{prop}
\begin{proof}
By \eqref{A8}, if $H(r)>0$,
\[\frac{d}{dr}\log H(r)=\frac{2N(r)}{r}\leq \frac{2d}{r}.\]
This means $r^{-2d}H(r)$ is non-increasing in $r$. Consequently,
$H(r)>0$ for all $r\in(0,R)$, and \eqref{eq:h_monotone} is a direct
consequence of the monotonicity of $r^{-2d}H(r)$.
\end{proof}
\begin{rmk}
By this doubling property, we can prove that $\{\bar{u}_\infty=0\}$
has zero Lebesgue measure. In fact, more properties such as the unique
continuation property can be proved by this method, see \cite{C-L
2}.
\end{rmk}

By this doubling proeprty, if $N(R;x)\leq d<\alpha$, then for all
$r\in(0,R)$,
\[H(r;x)\geq Cr^{2d}.\]
However, if $\bar{u}_\infty(x)=0$, because $\bar{u}_\infty$ is
$C^{\alpha}$ continuous,
\[H(r;x)\leq Cr^{2\alpha}.\]
If $r$ small, this is a contradiction. In other words, $N(r;x)\geq
\alpha$ for any $r>0$.

Combining this fact with \eqref{A.001}, we see for any
$x\in\{\bar{u}_\infty=0\}$ and $r>0$, $N(r;x)\equiv\alpha$. By
Proposition \ref{monotonocity},
\[\bar{u}_\infty(x+y)=|y|^{\frac{2}{p+1}}\bar{u}_\infty(x+\frac{y}{|y|}).\]
In particular, $\{\bar{u}_\infty=0\}$ is a cone with respect to any
point in $\{\bar{u}_\infty=0\}$. This then implies that
$\{\bar{u}_\infty=0\}$ is a linear subspace of $\mathbb{R}^n$.
Assume $\{\bar{u}_\infty=0\}=\mathbb{R}^k$ for some $k<n$. (Note
that $\bar{u}_\infty$ is nontrivial, so $\{\bar{u}_\infty=0\}$
cannot be the whole $\mathbb{R}^n$.) If $k\leq n-2$,
$\{\bar{u}_\infty=0\}$ has zero capacity and then $\bar{u}_\infty$
is a harmonic function. Because $\bar{u}_\infty\geq0$, by the strong
maximum principle, either $\bar{u}_\infty>0$ everywhere or
$\bar{u}_\infty\equiv0$. Both of these two lead to a contradiction.

 If
$k=n-1$, assume $\{\bar{u}_\infty=0\}=\{x_1=0\}$. Then by the
Schwarz reflection principle, $\bar{u}_\infty=c|x_1|$ for some
constant $c>0$. This again contradicts the global
$\alpha-$H\"{o}lder continuity of $\bar{u}_\infty$ because
$\alpha<1$.

\end{document}